\documentclass{amsart}
\usepackage{amssymb}
\usepackage[alphabetic]{amsrefs}

\newcommand\B{{\mathcal B}}
\renewcommand\b{\beta}
\newcommand\C{{\mathcal C}}
\renewcommand\c{\gamma}
\renewcommand\d{\delta}
\newcommand\F{{\mathbb F}}

\newcommand\N{{\mathbb N}}
\renewcommand\P{{\mathcal P}}

\newcommand\Z{{\mathbb Z}}

\newcommand\abs[1]{\left|#1\right|}
\newcommand\at[1]{\mathop{@}{#1}}
\newcommand\aut{\operatorname{Aut}}

\newcommand\rist{\operatorname{Rist}}
\newcommand\set[1]{\left\{#1\right\}}

\newcommand\stab{\operatorname{Stab}}

\newcommand\sym{\operatorname{Sym}}

\newtheorem{theorem}{Theorem}[section]
\newtheorem{proposition}[theorem]{Proposition}
\newtheorem{corollary}[theorem]{Corollary}
\newtheorem{lemma}[theorem]{Lemma}

\theoremstyle{remark}
\newtheorem{remark}[theorem]{Remark}
\theoremstyle{definition}
\newtheorem{definition}[theorem]{Definition}

\numberwithin{equation}{section}

\title{The Twisted Twin of the Grigorchuk Group}
\author{Laurent Bartholdi}
\address{Mathematisches Institut\\
  Georg-August Universit\"at zu G\"ottingen\\
  Bunsenstra\ss e 3--5\\
  D-37073 G\"ottingen}
\email{laurent.bartholdi@gmail.com}
\author{Olivier Siegenthaler}
\address{Departement Mathematik\\
R\"amistrasse 101\\
CH-8092 Z\"urich}
\email{olivier.siegenthaler@math.ethz.ch}

\begin{document}
\bibliographystyle{alpha}
\begin{abstract}
  We study a twisted version of Grigorchuk's first group, and stress
  its similarities and differences to its model.

  In particular, we show that it admits a finite endomorphic
  presentation, has infinite-rank multiplier, and does not have the
  congruence property.
\end{abstract}
\maketitle
\section{Introduction}
The first Grigorchuk group $\Gamma$ appeared
in~\cite{grigorchuk:burnside} in the early 1980's, as an example of a
finitely generated, infinite, torsion group. It was soon shown to be
just-infinite, and to have intermediate word
growth~\cite{grigorchuk:growth}, answering Milnor's
question~\cite{milnor:5603}. To this date it remains the fundamental
example of such a group.

The usual description of $\Gamma$ is as a ``group generated by
automata''. As such, it acts on a binary rooted tree, and its elements
are described by finite-state transducer automata. One may close
$\Gamma$ using the topology of the automorphism group of the binary
rooted tree; Sunic and Grigorchuk asked in 2007, in Cardiff, whether
the closure $\overline\Gamma$ of $\Gamma$ contained other finite-state
transducer automata, and in particular other with only one non-trivial
cycle (so-called ``bounded automata'').

In answering that question, the second author discovered a ``twisted''
version $G$ of $\Gamma$, and the purpose of the present paper is to
study some of its algebraic properties. In summary, $G$ also acts on
the infinite binary tree, and shares many properties with $\Gamma$:
\begin{itemize}
\item $G$ is torsion, just-infinite, and has the same closure as
  $\Gamma$ in the group of automorphisms of the tree;
\item $G$ has a finite endomorphic presentation
  \[G=\langle a,\b,\c,\d\mid a^2,\b^2,\c^2,\d^2,\varphi^n(R)\text{ for all }n\ge0\rangle,\]
  where
  $R=\{[\d^a,\d],[\d,\c^a\b],[\d,(\c^a\b)^\c],[\c^a\b,\c\b^a]\}$
  and $\varphi$ is the endomorphism of the free group on $a,\b,\c,\d$ induced by
  \begin{align*}
    a&\mapsto \c^a,&\b&\mapsto \d,&\c&\mapsto \b^a,&\d&\mapsto \c,
  \end{align*}
  see Theorem~\ref{thm:presentation}. Moreover, its relators are
  independent in $H_2(G,\Z)$ which, qua $(\Z/2\Z)[\varphi]$-module, is
  freely generated by $R$, see Theorem~\ref{thm:schur};
\item $G$ has bounded width; more precisely, if $(\gamma_n)$ denote
  its lower central series, then $\gamma_n/\gamma_{n+1}$ is an
  elementary Abelian $2$-group of rank $2$ if $\frac34 2^i+1\le n\le
  2^i$ for some $i$, and of rank $4$ otherwise, see
  Corollary~\ref{cor:lcs};
\item $G$ and $\Gamma$ are not isomorphic; $\Gamma$ does not have the
  congruence property, and more precisely its congruence kernel is
  isomorphic to $(\Z/4\Z)[[\{0,1\}^\omega]]$ as a profinite
  $\overline{G}$-module.
\end{itemize}

This paper is related to~\cite{bartholdi-s-z:cspbg} which studies more
systematically the congruence problem for groups generated by automata.

Another purpose of this paper is to assemble, in a systematic manner,
the various tools and techniques used to study a new example of group
generated by automata. We believe that the very computational nature
of these groups lends itself to computer calculations, and we make
liberal use of GAP~\cite{gap4:manual} to complete several steps.

We use the Knuth-Bendix procedure implemented in the package
KBMAG~\cite{holt:kbmag} to improve a presentation of our group.  Many
proofs reduce to computations in finite permutation groups; they are
easily performed in GAP.  Some other proofs are more involved, often
relying on an induction process.  The induction basis boils down to
verifications in nilpotent quotients of L-presented groups and
happily, such quotients are now efficiently implemented in the package
NQL~\cite{hartung:nql}.

We insist however that these computations could very well have been
done, and written down, by hand, at the cost of trebling the length of
this paper; and that the statements of this paper are true and not
empirical facts.

The methods presented here can all be applied with only slight
modifications to the Grigorchuk group, or to other groups generated by
automata, yet to come under the spotlight. We therefore include the
GAP command sequences we have used in an appendix.

\section{Basic Definitions and Notation}
We introduce the necessary vocabulary describing groups acting on
rooted trees. The main notions are summarized in Table~\ref{table:1}.

We consider the alphabet $X=\{0,1\}$ and we identify the free monoid $X^*$ over $X$ with the vertices of an infinite binary tree.
We let $\aut X^*$ denote the automorphism group of the tree $X^*$.

Let $v$ be a vertex of $X^*$.  The \emph{stabilizer} of $v$ is the group
$\stab(v)$ of elements of $\aut X^*$ which fix $v$.  The \emph{$n$-th level stabilizer}, written $\stab(n)$, is the subgroup of $\aut X^*$
consisting of the elements that fix all the vertices of the $n$-th
level.  It is a normal subgroup of $\aut X^*$.
The quotient $\aut X^n=\aut X^*/\stab(n)$ is a finite $2$-group isomorphic to the $n$-th fold iterated wreath product of $\sym(X)$, a cyclic group of order $2$.
Moreover, the group $\aut X^*$ is the projective limit of the projective system
$\aut X^n$, and hence $\aut X^*$ is a profinite group.
A basis of neighbourhoods of the identity is given by the collection $\{\stab(n):n\ge0\}$.

For $n\ge0$, we define the map
\begin{align*}
\psi_n:\aut X^*&\to\aut X^*\wr\aut X^n\\
g&\mapsto(g\at{v}\mid v\in X^n)\sigma.
\end{align*}
We shall often write $\psi$ for $\psi_1$.
The above defines the symbol $g\at{v}$. Namely, $g\at{v}$ is the projection of $g$ onto the $v$-th coordinate under $\psi_n$, where $n$ is the length of $v$.
A subset $S\subseteq\aut X^*$ is
called \emph{self-similar} if $S\at{v}$ is contained in $S$ for all
$v\in X^*$.
For a vertex $v\in X^*$ and an automorphism $g\in\aut X^*$, we let $v*g$ denote the automorphism satisfying
\begin{align*}
\psi_{|v|}(v*g)=(1,\dots,1,g,1,\dots,1),
\end{align*}
so that $(v*g)\at v=g$.

Consider a group $G\le\aut X^*$.
The group $G$ is \emph{level-transitive} if
the action of $G$ is transitive on each level $X^n$.
The group $G$ is \emph{recurrent} if $\stab_G(v)\at{v}=G$ for all $v\in X^*$.
A \emph{branching subgroup} is a subgroup $K\le G$ so that $X^n*K=\prod_{v\in X^n}v*K$ is a subgroup of $K$ for all $n\ge 0$.
The group $G$ is \emph{weakly regular branch} if it has a non-trivial branching subgroup.
The group $G$ is \emph{regular branch} if it is weakly regular branch and $X^n*K$ has finite index in $G$ for all $n\ge0$.
In such a situation we say that $G$ is regular branch over $K$.

\begin{table}[h]
\begin{center}
\begin{tabular}[c]{|c|ll|p{80mm}|}
\hline
symbol&\multicolumn{2}{l|}{name}&definition\\
\hline
\rule{0mm}{2.5ex}
$X^*$ & \multicolumn{2}{l|}{infinite binary tree} & \\
$X^n$ & \multicolumn{2}{l|}{$n$-th level of $X^*$} & \\
$\psi_n$ & \multicolumn{2}{l|}{wreath decomposition} & canonical map $\aut X^*\to\aut X^*\wr\aut X^n$\\
$g\at{v}$ & \multicolumn{2}{l|}{state of $g$ at $v$} & projection of $g$ on the $v$-th coordinate in $\psi_n(\aut X^*)$\\
$v*g$ & \multicolumn{2}{l|}{$v$-translate of $g$} & element of $\rist(v)$ with $(v*g)\at{v}=g$\\
$K$ & \multicolumn{2}{l|}{branching subgroup} & subgroup with $v*K\leq K$ for all $v\in X^*$\\
$X^n*K$ & & & product of $v*K$ for all $v\in X^n$\\
\hline
\rule{0mm}{2.5ex}
$\stab_G(v)$ & \multicolumn{2}{l|}{stabilizer of $v$} & subgroup of $G$ consisting of the elements which fix $v$\\
$\stab_G(n)$ & \multicolumn{2}{l|}{$n$-th level stabilizer} & intersection of $\stab_G(v)$ with $v$ ranging over $X^n$\\
\hline
\multicolumn{2}{|l|}{\rule{0mm}{2.5ex}level-transitive} & \multicolumn{2}{l|}{the action of $G$ is transitive on each level}\\
\multicolumn{2}{|l|}{self-similar} & \multicolumn{2}{l|}{the image of $G\to\aut X^*\wr\sym(X)$ is contained in $G\wr\sym(X)$}\\
\multicolumn{2}{|l|}{recurrent} & \multicolumn{2}{l|}{$\stab_G(v)\at{v}=G$ for all $v\in X^*$}\\
\multicolumn{2}{|l|}{weakly regular branch} & \multicolumn{2}{p{100mm}|}{$G$ contains a non-trivial branching subgroup}\\
\multicolumn{2}{|l|}{regular branch} & \multicolumn{2}{p{100mm}|}{weakly regular branch, and $X^n*K$ has finite index in $G$ for all $n\ge0$}\\
\hline
\end{tabular}
\end{center}
\caption{Symbols, subgroups, and main properties of groups acting on rooted trees.}
\label{table:1}
\end{table}

\section{Definition of the Group}

We define recursively the following elements of $\aut X^*$:
\begin{align*}
a&=\sigma,& \psi(\b)&=(\c,a),& \psi(\c)&=(a,\d),& \psi(\d)&=(1,\b),
\end{align*}
where $\sigma$ is the automorphism which permutes the two maximal subtrees.
We let $G$ be the group generated by $a,\b,\c,\d$.
We shall call this group the \emph{twisted twin} of the Grigorchuk group.
The motivation for this terminology is that the automaton defining $G$ is very similar to the one defining the Grigorchuk group, and moreover both groups have identical closure in $\aut X^*$ (see~\cite{siegenthaler:phd}).
However, $G$ is not isomorphic to the Grigorchuk group as we shall see.

Let $N$ be the generating set $\{1,a,\b,\c,\d\}$ of $G$.
Clearly $N$ is self-similar, and therefore so is $G$.
Next, $G$ acts transitively on $X$, and it is recurrent on the first level: $\stab_G(x)\at{x}=G$ for all $x\in X$.
Therefore $G$ is level-transitive and recurrent.
In particular, this implies that $G$ is infinite.

Define the group $K\leq G$ by
\begin{align*}
K=\langle[a,\b],[\b,\c],[\b,\d],[\c,\d],\b\c\d\rangle^G
\end{align*}
where the exponent stands for normal closure in $G$.

\begin{proposition}\label{prop:get branch}
The group $G$ is regular branch over $K$.
Moreover, $K$ contains $\stab_G(3)$.
\end{proposition}
\begin{proof}
Write $B=\langle\b,\c\d\rangle^G$.
The group $G/B$ is generated by the images of $a$ and $\d$, and therefore has order at most $\abs{\langle a,\d\rangle}=8$.

The group $B/K$ is generated by the image of $\b$, hence it is of order at most $2$.
Therefore $K$ is of index at most $16$ in $G$.
A direct computation shows that $K/\stab_K(3)$ has index $16$ in $G/\stab_G(3)$.
This proves that $K$ has index $2$ in $B$, that $B$ has index $8$ in $G$, and that $K$ contains $\stab_G(3)$.

We now prove that $X*K$ is contained in $K$.
For this we simply show that $1*g$ is contained in $K$ for all $g\in\set{[a,\b],[\b,\c],[\b,\d],[\c,\d],\b\c\d}$:
\begin{align*}
[\b,\d]&=1*[a,\b],&[\d,\b]\cdot[\d,[a,\b]]^\b&=1*[\b,\c],\\
[\d,\c]&=1*[\b,\d],&[[a,\b],\c]^\b\cdot[\b,\c]&=1*[\c,\d],\\
([a,\b]\b\c\d)^\d=\d\b^a\c&=1*\b\c\d.
\end{align*}
The last thing to prove is that $X^n*K$ has finite index in $K$.
Since $G$ is self-similar, $\stab_G(n+3)$ is a subgroup of $X^n*\stab_G(3)$.
Now $X^n*K$ contains $X^n*\stab_G(3)$, and therefore it also contains $\stab_G(n+3)$, which has finite index in $G$.
\end{proof}
\begin{remark}
Further computations show that the Abelianization of $K$ is $C_2\times(C_4)^3\times C_8$ and that $K/(X*K)$ is a cyclic group of order $4$.
\end{remark}

\section{A Presentation}
Recall that an (invariant) L-presentation is an expression of the form
\begin{align*}
\langle X\mid\Phi\mid R\rangle
\end{align*}
where $X$ is a set, $\Phi$ is a collection of endomorphisms of the free group $F_X$ with basis $X$, and $R$ is a set of elements of $F_X$.
The group defined by this L-presentation is then $\langle X\mid\{\phi(r)\}\rangle$ where $\phi$ ranges over the monoid generated by $\Phi$, and $r$ ranges over $R$.
Obviously, each endomorphism in $\Phi$ induces an endomorphism of the L-presented group.

L-presentations were introduced in~\cite{bartholdi:lpres}; a
fundamental example was Lysionok's presentation of the Grigorchuk
group~\cite{lysionok:pres}.

The main result of this section is the following
\begin{theorem}\label{thm:presentation}
The group $G$ admits the L-presentation
\begin{equation*}
G=\langle
a,\b,\c,\d\mid\tilde\varphi\mid a^2,\b^2,\c^2,\d^2,[\d^a,\d],[\d,\c^a\b],[\d,(\c^a\b)^\c],[\c^a\b,\c\b^a]
\rangle,
\end{equation*}
where $\tilde\varphi$ is the endomorphism of the free group on $a,\b,\c,\d$ induced by
\begin{align*}
a&\mapsto \c^a,&\b&\mapsto \d,&\c&\mapsto \b^a,&\d&\mapsto \c.
\end{align*}
\end{theorem}

Let $\varphi$ be the following map, defined on the generators of $G$:
\begin{align*}
a&\mapsto \c^a=\psi^{-1}(\d,a),&\b&\mapsto \d=\psi^{-1}(1,\b),\\
\c&\mapsto \b^a=\psi^{-1}(a,\c),&\d&\mapsto \c=\psi^{-1}(a,\d).
\end{align*}
\begin{proposition}\label{prop:endo}
The map $\varphi$ extends to an endomorphism of $G$.
\end{proposition}
\begin{proof}
From the above it is enough to prove that the map defined by
\begin{align*}
a&\mapsto\d,&\b&\mapsto1,&\c&\mapsto a,&\d&\mapsto a
\end{align*}
extends to a homomorphism $G\to G$.
The group generated by $a$ and $\d$ is a dihedral group of order $8$, and one can check that the map above factors through $G\to G/\stab_G(3)$.
\end{proof}
It is immediate that the restriction of $\varphi$ to $K$ sends $x$ to $\psi^{-1}(1,x)$.
We shall now prove Theorem~\ref{thm:presentation} and then some results which follow from this theorem.

\begin{definition}\label{def:contraction}
A self-similar group $G\leq\aut X^*$ is \emph{contracting} if there is a finite set $N\subseteq G$ such that for every $g\in G$ there is an integer $n$ such that $g\at{w}\in N$ for all $w\in X^*$ of length at least $n$.
The minimal set $N$ with this property is the \emph{nucleus} of $G$.
\end{definition}
\begin{lemma}\label{lem:get-contract}
The group $G$ is contracting with nucleus $N=\{1,a,\b,\c,\d\}$.
\end{lemma}
\begin{proof}
There are several ways of proving that $G$ is contracting.
The most direct one is to notice that $N$ generates $G$; and that, if
$g$ is in $N^2$, then $g\at{w}$ is in $N$ for all $w$ of length at
least $3$.

It is then easy to prove that $N$ is the nucleus of $G$, by induction on the length of the elements (see~\cite{nekrashevych:ssg} where this is done in detail, and in greater generality).
\end{proof}

We recall the general strategy in obtaining presentations by generators
and relations for self-similar groups; for more details
see~\cite{bartholdi:lpres} or \cite{sidki:pres}.

The only relations of length $\le3$ among elements of $N$ are
$a^2=\b^2=\c^2=\d^2=1$.  We thus consider the group $F=\langle
a,\b,\c,\d\mid a^2,\b^2,\c^2,\d^2\rangle$. The decomposition map
$\psi:G\to G\wr\sym(X)$ restricts to a map $N\to N^X\times\sym(X)$,
which induces a homomorphism $\bar\psi:F\to F\wr\sym(X)$. Thanks to
the relations of length $\le3$ we took in $F$, the recursion
$\bar\psi$ is also contracting, with nucleus $\{1,a,\b,\c,\d\}$; see
again~\cite{nekrashevych:ssg} for details. Set $R_0=1\triangleleft F$
and $R_{n+1}=\psi^{-1}(R_n^X)$ for all $n\ge0$.
\begin{lemma}
$G=F/\bigcup_{n\ge0}R_n$.
\end{lemma}
\begin{proof}
  By our choice of relations in $F$, the decomposition $\bar\psi$ is
  contracting on $F$, with nucleus $\{1,a,\b,\c,\d\}$. Given $w\in F$: if
  $w$ belongs to $R_n$ for some $n$, then it is clear that $w$ is
  trivial in $G$. Conversely, if $w$ is trivial in $G$,
  there exists $n$ such that all $w$'s level-$n$ states belong to $N$
  and act trivially; so they are all $1$; so $w$ belongs to $R_n$.
\end{proof}

Using the Reidemeister-Schreier rewriting method, we can
construct a normal generating set for $R_1$.  Indeed, $\bar\psi$ induces
an injective map $F/R_1\to F\wr\sym(X)$, whose image has infinite index.
However, we may chose the infinite dihedral group generated by $a$ and $\d$ on the first copy of $F$ as a transversal for $\bar\psi(F)$ in $F\wr\sym(X)$.
Then we obtain quite explicitly
\[R_1=\langle [\d^a,\d^w],[\d^a,(\c\b^a)^w],[(\c\b^a)^a,(\c\b^a)^w]\rangle^F,\]
where $w$ ranges over the elements of the infinite dihedral group generated by $\c$ and $\c^a$.

We now consider the homomorphism $\bar\varphi:F\to F$ defined by
\begin{align*}
a&\mapsto \c^a,&\b&\mapsto \d,&\c&\mapsto \b^a,&\d&\mapsto \c.
\end{align*}
By Proposition~\ref{prop:endo}, for all normal generators $r$ of $R_1$ we have
$\bar\psi(\bar\varphi(r))=(r',r)$ for some relation $r'\in\langle a,\d\rangle$.
It happens that $(a\d)^4$ is in  $R_1$ and that $\langle a,\d\rangle$ is a dihedral group of order $8$ in $G$.
Therefore $r'$ is in $R_1$ and $R_n$ is normally generated by
$\bigcup_{i<n}\bar\varphi^i(R_1)$ for all $n>0$. We conclude:
\begin{proposition}
  \[G=\langle
  a,\b,\c,\d\mid a^2,\b^2,\c^2,\d^2,\tilde\varphi^n(r)\text{
    for all }n\ge0\rangle,\]
  where $r$ ranges over the normal generators of $R_1$ above.
\end{proposition}

\begin{proof}[Proof of Theorem~\ref{thm:presentation}]
One of the $r$ above is $[\d^a,\d]$.
Applying $\tilde\varphi$ we get $\tilde\varphi([\d^a,\d])\equiv(\c\c^a)^4=1$.
Therefore it is enough to consider 8 values of $w$, for example the following ones:
\begin{align*}
w\in\{1,\c,\c^a,\c\c^a,\c^a\c,\c\c^a\c,\c^a\c\c^a,\c\c^a\c\c^a\}.
\end{align*}
Also, $\tilde\varphi(s)$ is clearly a consequence for $s$ in $\{a^2,\b^2,\c^2,\d^2\}$.
Thus we proved the following
\begin{align}\label{eq:pres}
G=\langle
a,\b,\c,\d\mid\tilde\varphi\mid a^2,\b^2,\c^2,\d^2,[\d^a,\d^w],[\d^a,(\c\b^a)^w],[\c\b^a,(\c\b^a)^w]\rangle,
\end{align}
with $w$ ranging over the set above.
Then we use a computer to prove that the normal closure of the relations in the claim together with their iterates under $\tilde\varphi$ contains all the relations of the presentation~\eqref{eq:pres}.
This is done using GAP, in Lemma~\ref{lem:KB}.
\end{proof}

\begin{remark}
In the same way, one can also recover the usual presentation of the Grigorchuk group (see~\cite{lysionok:pres}):
\begin{align}\label{eq:lpres Grigorchuk}
\langle a,b,c,d\mid\phi\mid a^2,b^2,c^2,d^2,bcd,
[d^a,d],[d^a,d^{c^ac}]\rangle
\end{align}
where $\phi$ is given by
\begin{align*}
a\mapsto c^a,&&b\mapsto d,&&c\mapsto b,&&d\mapsto c.
\end{align*}
\end{remark}

Knowing a presentation (even an infinite one) is very useful in practice.
For example, this makes the following computations straightforward.
\begin{proposition}\label{prop:K contains gamma_3}
The group $K$ contains $\gamma_3(G)$.
\end{proposition}
\begin{proof}
Using Theorem~\ref{thm:presentation}, we see that $G/K$ is (a quotient of) the group presented by
\begin{align*}
\langle a,\c,\d\mid a^2,\c^2,\d^2,[\c,\d],[a,\c\d],[\d,\d^a],[\c,\d^a],[\c,\c^a]\rangle.
\end{align*}
A Todd-Coxeter procedure shows that this group has order $16$ and is nilpotent of class $2$.
\end{proof}

Define the group
\begin{align*}
A = \langle a\rangle^G
\end{align*}
\begin{proposition}
The group $A$ contains $\gamma_3(G)$ and has index $16$ in $G$.
\end{proposition}
\begin{proof}
From Theorem~\ref{thm:presentation} we see that the group $G/A$ admits the presentation
\begin{align*}
\langle\b,\c,\d\mid\b^2,\c^2,\d^2,[\d,\c\b],[\c,\b\d],[\b,\d\c]\rangle.
\end{align*}
A Todd-Coxeter procedure shows that this group has order $16$ and is nilpotent of class $2$.
\end{proof}

The following Proposition is needed in the article~\cite{bartholdi-s-z:cspbg}.
\begin{proposition}\label{prop:[K,G]}
The group $[K,G]$ has index $128$ in $G$ and is generated by $\gamma_3(G)$ and $[a,\b\c\d]$.
Moreover, $[K,G]$ contains $\gamma_3(G)$ and the quotient $K/[K,G]$ is isomorphic to $C_4\times C_2$, where $C_4$ is generated by $\b\c\d$ and $C_2$ is generated by $[a,\b]$.
\end{proposition}
\begin{proof}
Since $K$ contains $\gamma_3(G)$, the group $[K,G]$ contains $\gamma_4(G)$.
Therefore, all computations can be made in $G/\gamma_4(G)$, which is a finite $2$-group.
The rest is routine check using NQL~\cite{hartung:nql}.
\end{proof}

Let $C$ be the group generated by $[K,G]$ and $[a,\b][\b,\c]$.
\begin{proposition}\label{prop:C}
The group $C$ is normal, of index $64$ in $G$.
Moreover, $C$ contains $\gamma_3(G)$ and the quotient $K/C$ is isomorphic to $C_4$, generated by $\b\c\d$.
\end{proposition}
\begin{proof}
Since $[K,G]$ contains $\gamma_3(G)$, all computations can be made in $G/\gamma_3(G)$, which is a finite $2$-group.
The rest is routine check using NQL~\cite{hartung:nql}.
\end{proof}

\section{The Congruence Kernel}

\subsection{A Basis of Neighbourhoods of the Identity in $\widehat{G}$}
\begin{definition}
Let $X$ be a set and let $\{A_i\}_{i\in I}$ and $\{B_j\}_{j\in J}$ be two filters in $X$.
We say $\{A_i\}$ is \emph{cofinal to $\{B_j\}$} if for every $j\in J$, there is $i\in I$ such that the inclusion $A_i\leq B_j$ holds.
We say $\{A_i\}$ and $\{B_j\}$ are \emph{cofinal} if $\{A_i\}$ is cofinal to $\{B_j\}$ and $\{B_j\}$ is cofinal to $\{A_i\}$.
\end{definition}

\begin{definition}
Let $G$ be a subgroup of $\aut X^*$.
A subgroup $H\leq G$ is called a \emph{congruence subgroup} if $H$ contains $\stab_G(n)$ for some $n\ge0$.
\end{definition}

Consider the three collections $\C=\{\stab_G(n):n\ge0\}$, $\B=\{X^n*K:n\ge0\}$ and $\P=\{N\le G:N\text{ is normal of finite index in }G\}$.
Taken as basis of neighbourhoods of the identity in $G$, they define the \emph{congruence}, \emph{branch} and \emph{profinite} topology, respectively.

We already know that $\C$ and $\B$ are cofinal by Proposition~\ref{prop:get branch}.
That is, every $X^n*K$ is a congruence subgroup.
This implies that the congruence and the branch topologies coincide, and therefore the rigid kernel of $G$ is trivial, in the terminology introduced in~\cite{bartholdi-s-z:cspbg}.

The completion of $G$ with respect to $\C$ is $\overline{G}=\varprojlim G/\stab_G(n)$;
this group is the closure of $G$ in $\aut X^*$.
The completion of $G$ with respect to $\P$ is $\widehat{G}=\varprojlim_\P G/N$, this is the \emph{profinite completion} of $G$.

Recall the group $C$ from Proposition~\ref{prop:C}.
We now prove that the collection $\P$ and the sequence $\{X^n*C:n\in\N\}$ are cofinal.
We shall see later that the groups $X^n*C$ are \emph{not} congruence subgroups.
Therefore the branch kernel of $G$ is non-trivial.

\begin{proposition}\label{prop:C cofinal}
Any non-trivial normal subgroup of $G$ contains $X^n*C$ for all $n$ big enough.
\end{proposition}
\begin{corollary}
$G$ is just-infinite.
\end{corollary}

The proof of the proposition splits into three lemmas.
The following has been extracted from the proof of~\cite{grigorchuk:jibg}*{Theorem~4}.
\begin{lemma}\label{lem:branch K'}
Let $G\leq\aut X^*$ be a level-transitive, self-similar, regular branch group over $K$.
Then any non-trivial normal subgroup of $G$ contains $X^n*K'$ for all $n$ big enough.
\end{lemma}

Since the above lemma applies for $G$, the next two prove Proposition~\ref{prop:C cofinal}.

\begin{lemma}\label{lem:K' [K,G]}
The group $K'$ contains $X^n*[K,G]$ for all $n$ big enough.
\end{lemma}
\begin{proof}
We shall prove the statement for $n=5$.
Since $K'$ contains $K'\times K'$, the claim follows.

$K'$ is characteristic in $K$ which is normal in $G$, thus $K'$ is normal in $G$.
Hence it is sufficient to show that $1^5*[k,s]$ is in $K'$ for all $k\in K$ and $s\in\{a,\b,\c,\d\}$.
We define the following elements of $K$:
\begin{align*}
g_a&=[a,\b]^{4\c\b^a},&g_\b&=[\c,\d]^4,\\
g_\c&=[\b,\d]^{4\b^a},&g_\d&=[a,\b]^{4\c}.
\end{align*}
Then $g_s\in\stab_K(1^5)$ and $(g_s)\at{1^5}=s$ for all $s\in\{a,\b,\c,\d\}$.
Therefore $[1^5*k,g_s]=1^5*[k,s]$ is in $K'$ for all $k\in K$ and $s\in\{a,\b,\c,\d\}$.
\end{proof}

\begin{lemma}\label{lem:[K,G] C}
The group $[K,G]$ contains $X*C$.
\end{lemma}
\begin{proof}
Since $C$ is a subgroup of $K$, and $C$ is normal in $G$ and generated by $[K,G]$ and $x=[a,\b][\b,\c]$, it is enough to show that $1*x=\varphi(x)$ is contained in $[K,G]$.
One can check that $1*x$ is trivial in $G/\gamma_3$.
\end{proof}

\subsection{The Kernel}
The congruence kernel of $G$ is non-trivial, and therefore $G$ does not have the congruence subgroup property:
\begin{theorem}
The congruence kernel $\ker(\widehat{G}\to\overline{G})$ is isomorphic to $(\Z/4\Z)[[X^\omega]]$ as a profinite $\overline{G}$-module.
\end{theorem}
\begin{proof}
We use the method presented in~\cite{bartholdi-s-z:cspbg} to compute the kernel.
It is proved that one can compute the congruence kernel as $A[[\overline{G}/\overline{G_w}]]$, where $A$ is the finite group $\bigcap_{n\ge0}\varphi^n(K/C)$ and $\overline{G_w}$ is the stablizer of a point of the boundary of the tree in $\overline{G}$.

$K/C$ is a cyclic group of order $4$, and $\varphi$ induces the non-trivial automorphism of this group.
Therefore $A\simeq\Z/4\Z$.
Next, the action of $G$ by conjugation on $K/[K,G]$ is clearly trivial.
Since $K/C$ is a quotient of that group, $G$ acts trivially on it.
Therefore the action of $\overline{G}$ on the congruence kernel is given by the canonical action on the boundary $X^\omega$ of $X^*$, and the kernel is $(\Z/4\Z)[[X^\omega]]$.
\end{proof}

\section{Germs}

We construct a homomorphism from $G$ to a finite group, which does not factor through $\aut X^n$ for any $n$.
More precisely, we give explicit functions $G\to\F_2$, where $\F_2$ is the finite field with two elements (endowed with the discrete topology), which are continuous for the profinite topology on $G$, but not for the congruence topology.
This yields some insights about the germs of the action on the boundary of the tree.

Consider the group $C$ of Proposition~\ref{prop:C} and write
$\Gamma=G/C$, a group of order $2^6$. Define, for all $n\in\N$, the
homomorphism $\pi_n:G\to\Gamma\wr\aut X^n$ as $\psi:G\to G\wr\aut X^n$
followed by the natural quotient map.

\begin{theorem}\label{thm:germs}
  Every epimorphism from $G$ onto a finite group factors through
  $\pi_n$ for all $n$ big enough.
\end{theorem}
The explicit computation of the maps $\pi_n$ is a bit technical, and shall be given after the proof of Theorem~\ref{thm:G->Gamma}.
Recall that $N=\{1,a,\b,\c,\d\}$ is the nucleus of $G$; see Definition~\ref{def:contraction}.

\begin{lemma}\label{lem:contract}
For all $g,h\in G$, there is an integer $n$ such that $g\at{w}$, $h\at{w}$ and $(gh)\at{w}$ all belong to $N$ for all words $w$ of length at least $n$.
\end{lemma}
\begin{proof}
By definition there is $n_g$ such that $g\at{w}$ is in $N$ for all $w$ of length at least $n_g$.
Similarly there is $n_h$ for $h$ and $n_{gh}$ for $gh$.
The claim obviously holds with $n=\max\{n_g,n_h,n_{gh}\}$.
\end{proof}

Define the permutation $\tau$ of the set $\{\b,\c,\d\}$ as
\begin{align*}
\b\mapsto\c,&&\c\mapsto\d,&&\d\mapsto\b,
\end{align*}
so that $\tau^3$ is the identity.
To each $\xi\in\{\b,\c,\d\}$ we associate a function $G\to\F_2$ defined by
\begin{align*}
[\xi](g)=\begin{cases}
1&\text{if }g = \xi,\\
0&\text{if }g\neq\xi\text{ but }g\in N,\\
\sum_{x\in X}[\xi^\tau](g\at{x})&\text{otherwise}.
\end{cases}
\end{align*}
It is clear that the relation
\begin{align}\label{eq:xi}
[\xi](g)=\sum_{w\in X^n}[\xi^{\tau^n}](g\at{w})
\end{align}
holds for all $\xi\in\{\b,\c,\d\}$, $g\in G$, and all $n\geq0$.

The restriction of the function $[\xi]$ to $N$ is the characteristic function of $\xi$.
We now prove that $[\xi](g)$ is the number (modulo $2$) of occurrences of $\xi$ in a decomposition of $g$ as a product of elements of $N$.
This statement is made precise in the following lemma.
\begin{lemma}\label{lem:homomorphisms}
Let $\xi$ be one of $\b,\c,\d$.
Then $\xi:G\to\F_2$ is a well-defined group homomorphism.
\end{lemma}
\begin{proof}
The group $G$ is contracting, therefore the computation of $[\xi](g)$ only involves a finite number of steps; hence $[\xi]$ is well-defined.

Note that another way to compute $[\xi](g)$ is as follows.
Choose a multiple $n$ of $3$ such that $g\at{w}$ is in $N$ for all $w\in X^n$.
Then $[\xi](g)$ is the number (modulo $2$) of words $w$ of length $n$ such that $g\at{w}=\xi$.
This is equivalent to the original definition of $[\xi]$ by Equation~\eqref{eq:xi}.

The fact that $[\xi]$ is a group homomorphism follows from the above observation.
Because $N$ generates $G$ it is sufficient to show that for all $s\in N$ and $g\in G$ the relation $[\xi](sg)=[\xi](s)+[\xi](g)$ holds.
In other words we must show that the following holds for all $s\in N$ and $g\in G$:
\begin{align*}
[\xi](sg)-[\xi](g)=\begin{cases}
1&\text{if }s=\xi,\\
0&\text{otherwise}.
\end{cases}
\end{align*}

Consider $s\in N$ and $g\in G$ and let $n$ be a multiple of $3$ big enough so that $s\at{w}$, $g\at{w}$ and $(sg)\at{w}$ are all in $N$ for all $w\in X^n$ (cf. Lemma~\ref{lem:contract}).
Note that we have the relation
\begin{align*}
(sg)\at{v}=(s\at{v})(g\at{v^s}).
\end{align*}
If $s\neq\xi$ it is clear that $s\at{w}$ is different from $\xi$ for all $w$ of length $n$, and therefore $[\xi](g)=[\xi](sg)$.
If $s=\xi$ then there is exactly one word $v$ of length $n$ such that $s\at{v}=\xi$.
Then $[\xi]((sg)\at{w})=[\xi](g\at{w^s})$ for all $w\neq v$ of length $n$, but we have $[\xi]((sg)\at{v})=[\xi](g\at{v^s})+1\pmod{2}$.
Therefore $[\xi](sg)=\sum_{w\in X^n}[\xi]((sg)\at{w})=[\xi](g)+1\pmod{2}$ in this case.
\end{proof}

We define the notation $[x\xi](g)=[\xi](g\at{x})$ for $x\in X$, $\xi\in\{\b,\c,\d\}$ and $g\in G$.
Also, we let $[\varnothing]:G\to\F_2$ be the canonical epimorphism $G\to G/\stab_G(1)$, after identifying $G/\stab_G(1)$ with the additive group of $\F_2$.

\begin{lemma}\label{lem:table1}
The following relations hold
\begin{equation*}
\begin{array}{|c|cccc|}
\hline
f&f(ag)-f(g)&f(\b g)-f(g)&f(\c g)-f(g)&f(\d g)-f(g)\\
\hline
[\varnothing]\cdot[\b]&[\b](g)&[\varnothing](g)&0&0\\
{[}\varnothing]\cdot[\c]&[\c](g)&0&[\varnothing](g)&0\\
{[}\varnothing]\cdot[\d]&[\d](g)&0&0&[\varnothing](g)\\
\hline
\end{array}
\end{equation*}
for all $g\in G$.
\end{lemma}
\begin{proof}
These are straightforward consequences of the definitions and of Lemma~\ref{lem:homomorphisms}.
For example, we have $([\varnothing]\cdot[\b])(ag)=([\varnothing](a)+[\varnothing](g))\cdot([\b](a)+[\b](g))=([\varnothing]\cdot[\b])(g)+[\b](g)$.
\end{proof}

\begin{lemma}
The relation
\begin{align*}
[x\xi](sg)=[x\xi](s)+[x\xi](g)+[\varnothing](s)\cdot[\xi^{\tau^{-1}}](g)
\end{align*}
holds for all $x\in X$, $\xi\in\{\b,\c,\d\}$ and $s,g\in G$.
\end{lemma}
\begin{proof}
We compute $[x\xi](sg)=[\xi]((sg)\at{x})=[\xi]((s\at{x})(g\at{x^s}))=[\xi](s\at{x})+[\xi](g\at{x^s})=[x\xi](s)+[x^s\xi](g)$.
The claim follows from the relation $[\xi](g)=\sum_{x\in X}[x\xi^\tau](g)$.
\end{proof}
In the following we write $\bar x=1-x$.

\begin{corollary}\label{cor:a}
The relations $[x\xi](ag)=[\bar x\xi](g)$ and $[x\xi](ag)-[x\xi](g)=[\xi^{\tau^{-1}}](g)$ hold for all $g\in G$.
\end{corollary}
\begin{corollary}\label{cor:bcd}
The relation $[x\xi](sg)=[x\xi](g)+[x\xi](s)$ holds for all $s\in\{\b,\c,\d\}$ and all $g\in G$.
\end{corollary}
\begin{corollary}\label{cor:table2}
The following relations hold
\begin{equation*}
\begin{array}{|c|cccc|}
\hline
f&f(ag)-f(g)&f(\b g)-f(g)&f(\c g)-f(g)&f(\d g)-f(g)\\
\hline
[1\c]&[\b](g)&0&0&0\\
{[}0\d]&[\c](g)&0&0&0\\
{[}0\b]&[\d](g)&0&0&0\\
\hline
\end{array}
\end{equation*}
for all $g\in G$.
\end{corollary}

Define recursively the following functions

\begin{align*}
\begin{split}
f_1(g)&=[0\d](g)+[1\c](g)+[0\d](g)\cdot [1\d](g)+\sum_{x\in X}([x\b](g)\cdot [\bar x\d](g)+f_2(g\at{x})),
\end{split}\\
f_2(g)&=[0\b](g)+[1\c](g)+[0\c](g)\cdot [1\c](g)
+\sum_{x\in X}f_3(g\at{x}),\\
\begin{split}
f_3(g)&=[0\d](g)+[1\c](g)+[0\c](g)\cdot [1\c](g)\\
&\quad+\sum_{x\in X}([x\b](g)\cdot [\bar x\c](g)+[x\b](g)\cdot [\bar x\d](g)+f_1(g\at{x})).
\end{split}
\end{align*}
\begin{lemma}
The functions $f_i:G\to\F_2$ are well-defined for all $i\in\{1,2,3\}$.
\end{lemma}
\begin{proof}
This follows again form the fact that $G$ is contracting.
Indeed, each of the functions $f_i$ consists of infinitely many terms.
On the other hand it is easy to see that $f_i(g)=0$ for all $g\in N$ because all the terms of these infinite sums vanish.
Therefore the evaluation of $f_i(g)$ only involves finitely many non-zero terms.
\end{proof}

\begin{proposition}\label{prop:f_i}
The following relations hold
\begin{equation*}
\begin{array}{|c|cccc|}
\hline
f&f(ag)-f(g)&f(\b g)-f(g)&f(\c g)-f(g)&f(\d g)-f(g)\\
\hline
f_1&[\b](g)+[\c](g)&[\c](g)+[\d](g)&0&[\c](g)\\
f_2&[\b](g)+[\d](g)&[\d](g)&[\b](g)+[\d](g)&0\\
f_3&[\b](g)+[\c](g)&0&[\b](g)&[\b](g)+[\c](g)\\
\hline
\end{array}
\end{equation*}
for all $g\in G$.
\end{proposition}
\begin{proof}
We begin with the first column, using Corollary~\ref{cor:a}.
The quadratic terms of each $f_i$ contribute to nothing in $f_i(ag)-f_i(g)$.
Moreover, $\sum_{x\in X}f_i((ag)\at{x})=\sum_{x\in X}f_i(\at{x^a})=\sum_{x\in X}f_i(\at{x})$, and therefore this term also gives no contribution.
The linear terms clearly yield the result.

We use Corollary~\ref{cor:bcd} to prove the other nine relations.
The proof goes by induction, using the fact that $G$ is contracting.
We compute $f_1(\b g)-f_1(g)$ as follows.
We note that $[x\xi](\b)=0$ unless $x=0$ and $\xi=\c$.
Therefore we have
\begin{align*}
f_1(\b g)-f_1(g)=f_2(\c\cdot g\at{0})-f_2(g\at{0})+f_2(a\cdot g\at{1})-f_2(g\at{1}).
\end{align*}
Now we already computed $f_2(a\cdot g\at{1})-f_2(g\at{1})=([\b]+[\d])(g\at{1})=([1\b]+[1\d])(g)$.
Thus we get
\begin{align*}
f_1(\b g)-f_1(g)&=[1\b](g)+[1\d](g)+f_2(\c\cdot g\at{0})-f_2(g\at{0})\\
&=[\d](g)+[\c](g)+f_2(\c\cdot g\at{0})-f_2(g\at{0})-[\b](g\at{0})-[\d](g\at{0}).
\end{align*}
This is exactly what we wanted, if the relation $f_2(\c g)-f_2(g)=[\b](g)+[\d](g)$ holds.
In the same way, one shows that the latter holds provided $f_3(\d g)-f_3(g)=[\b](g)+[\c](g)$ holds.
This, in turn, holds if $f_1(\b g)-f_1(g)=[\c](g)+[\d](g)$ holds, and we are back where we started.
For the other relations, the induction step follows in a similar same way.

Thus, we are proving all the relations in a cyclic manner.
For the basis of the induction it is enough to consider the case $g\in N$.
Better, we only need to compute $f_i(sg)-f_i(g)$ for all $s,g\in N$ so that $sg$ is in $N$ (cf. Lemma~\ref{lem:contract}).
Now $f_i(sg)-f_i(g)$ is obviously $0$ if $g$ is trivial, because $f_i(s)=0$ for all $s\in N$.
The only case left is $s=g$, and it is clear that $f_i(sg)-f_i(g)$ also vanishes in this situation.
\end{proof}

\begin{theorem}\label{thm:G->Gamma}
Let $\Gamma$ be the nilpotent group of order $2^6$ given by the following presentation:
\begin{align*}
\Gamma=\langle a,b,c,d\mid a^2,b^2,c^2,d^2,[a,b]^2,[a,c]^2,[a,b]=[b,c]=[c,d]=[d,b],[a,d]=[a,b][a,c]\rangle,
\end{align*}
and consider the map $\pi:G\to\Gamma$ defined by $\pi(g)=a^{[\varnothing](g)}b^{[\b](g)}c^{[\c](g)}d^{[\d](g)}[a,b]^{e(g)}[a,c]^{f(g)}$ with $e(g)=(f_3+[\varnothing]\cdot([\b]+[\d])+[0\b]+[0\d])(g)$ and $f(g)=([\varnothing]\cdot([\c]+[\d])+[0\b]+[0\d])(g)$.
Then $\pi$ is a surjective group homomorphism.
Moreover, the kernel of $\pi$ is $C$.
\end{theorem}
\begin{proof}
The identity is mapped to the identity under $\pi$.
Thus it is enough to show that $\pi$ is multiplicative.
Using Lemma~\ref{lem:table1}, Corollary~\ref{cor:table2} and Proposition~\ref{prop:f_i}, it is straightforward to check that we have the relations
\begin{equation*}
\begin{array}{|c|cccc|}
\hline
\phi&\phi(ag)-\phi(g)&\phi(\b g)-\phi(g)&\phi(\c g)-\phi(g)&\phi(\d g)-\phi(g)\\
\hline
e&0&[\varnothing](g)&[\b](g)&([\b]+[\c]+[\varnothing])(g)\\
f&0&0&[\varnothing](g)&[\varnothing](g)\\
\hline
\end{array}
\end{equation*}
for all $g\in G$.
Therefore $\pi$ is a group homomorphism, since $a,\b,\c,\d$ generate $G$ as a semigroup.

One can check that $\pi$ is surjective, and that $C$ is mapped to the trivial group.
Since $C$ has index $2^6$ in $G$, the theorem is proved.
\end{proof}

\begin{proof}[Proof of Theorem~\ref{thm:germs}]
By the previous theorem, the kernel of $\pi_n$ is $X^n*C$.
Any finite-index subgroup contains $X^n*C$ for all $n$ big enough, by Proposition~\ref{prop:C cofinal}.
\end{proof}

\section{The Lower Central Series}

Let $\gamma_1=G$ and $\gamma_{n+1}=[G,\gamma_n]$ be the terms of the lower central series of $G$.
\begin{proposition}\label{prop:lcs}
We have, for all $n\geq3$:
\begin{align*}
\gamma_{2n-1}&=\gamma_n\times\gamma_n,&
\gamma_{2n}&=\{(g,g^{-1}):g\in\gamma_n\}\gamma_{2n+1}.
\end{align*}
\end{proposition}
We first prove two lemmas and then the proposition.
Define $\bar{\gamma}_{2n-1}=\gamma_n\times\gamma_n$ and $\bar{\gamma}_{2n}=\{(g,g^{-1}):g\in\gamma_n\}\bar\gamma_{2n+1}$ for $n\geq3$.
The groups $\bar{\gamma}_k$ are subgroups of $G$, because $K$ contains $\gamma_3$ (see Proposition~\ref{prop:K contains gamma_3}).
Notice that the groups $\bar\gamma_k$ form a descending filtration in $G$.

\begin{lemma}\label{lem:lcs1}
We have the inclusion $[\bar\gamma_k,G]\subseteq\bar\gamma_{k+1}$ for $k\geq5$.
More precisely, for all $h\in\stab_G(1)$ we have $[\bar\gamma_k,h]\subseteq\bar\gamma_{k+2}$ and $[\bar\gamma_k,a]\subseteq\bar\gamma_{k+1}$ if $k$ is odd, and $[\bar\gamma_k,h]\subseteq\bar\gamma_{k+1}$ and $[\bar\gamma_k,a]\subseteq\bar\gamma_{k+2}$ if $k$ is even.
\end{lemma}
\begin{proof}
Suppose $k$ is odd and write $k=2n-1$.
Consider $g=(g_0,g_1)\in\gamma_n\times\gamma_n=\bar\gamma_k$ and $h=(h_0,h_1)\in\stab_G(1)$.
Then obviously $[g,h]=([g_0,h_0],[g_1,h_1])\in\bar\gamma_{2n+1}$ and $[g,a]=(g_0^{-1}g_1,(g_0^{-1}g_1)^{-1})\in\bar\gamma_{2n}$.

Suppose $k$ is even and write $k=2n$.
Consider $g\in\gamma_n$ and $h=(h_0,h_1)\in\stab_G(1)$.
Then obviously $[(g,g^{-1}),h]=([g,h_0],[g^{-1},h_1])$ is in $\bar\gamma_{2n+1}$.
Next, since $G$ is generated by involutions, the quotients in the lower central series are elementary Abelian $2$-groups.
This implies that $g^2\in\gamma_{n+1}$ and thus $[(g,g^{-1}),a]=(g^{-2},g^2)\in\bar\gamma_{2n+2}$.
Notice that we already proved the inclusions $[\bar\gamma_{2n+1},h]\subseteq\bar\gamma_{2n+1}$ and $[\bar\gamma_{2n+1},a]\subseteq\bar\gamma_{2n+2}$, so we are done.
\end{proof}

\begin{lemma}\label{lem:lcs2}
For all $n\geq3$, we have the inclusion $[\bar\gamma_{2n-1},G]\supseteq\bar\gamma_{2n}$; for all $n\geq3$, if the equality $\bar\gamma_k=\gamma_k$ holds for all $5\leq k\leq n+2$, then we have the inclusion $[\bar\gamma_{2n},G]\supseteq\bar\gamma_{2n+1}$.
\end{lemma}
\begin{proof}
It is clear that $[\bar\gamma_{2n-1},G]$ contains $[(1,g),a]=(g,g^{-1})$ for all $g\in\gamma_n$.
Since $G$ is recurrent, we have $[\bar\gamma_{2n-1},G]\supseteq\gamma_{n+1}\times\gamma_{n+1}=\bar\gamma_{2n+1}$.
Therefore $[\bar\gamma_{2n-1},G]$ contains $\bar\gamma_{2n}$ and we are done for the first statement.

Suppose now $\bar\gamma_k=\gamma_k$ holds for all $5\leq k\leq n+2$.
By definition, $\bar\gamma_{2n}$ contains $\bar\gamma_{2n+1}$.
Therefore $[\bar\gamma_{2n},G]$ contains $[\bar\gamma_{2n+1},G]\supseteq\bar\gamma_{2n+2}\supseteq\gamma_{n+2}\times\gamma_{n+2}$.
We now prove that, modulo $\gamma_{n+2}\times\gamma_{n+2}$, the group $[\bar\gamma_{2n},G]$ contains $([g,s],1)$ for all $g\in\gamma_n$ and all $s\in\{a,\b,\c,\d\}$, and this will conclude the proof.
A direct computation shows that $[\bar\gamma_{2n},G]$ contains the elements
\begin{align*}
[(g,g^{-1}),\b]&=([g,\c],[g^{-1},a]),&[(g,g^{-1}),\c]&=([g,a],[g^{-1},\d]),&[(g,g^{-1}),\d]&=(1,[g^{-1},\b])
\end{align*}
for all $g\in\gamma_n$.
Using Lemmas~\ref{lem:lcs1} and~\ref{lem:lcs GAP}, we see that either $[g,a]$ or $[g^{-1},\b],[g,\c],[g^{-1},\d]$ are in $\gamma_{n+2}$, so we are done.
\end{proof}

\begin{proof}[Proof of Proposition~\ref{prop:lcs}]
We need to prove the equality $\bar\gamma_k=\gamma_k$ for all $k\geq5$.
Note that if we prove this for all $5\leq k\leq n$, then Lemmata~\ref{lem:lcs1} and~\ref{lem:lcs2} show that $[\bar\gamma_k,G]=\bar\gamma_{k+1}$ for all $5\leq k\leq2n-3$.
In turn, this implies $\bar\gamma_k=\gamma_k$ for all $5\leq k\leq 2n-2$.
Therefore, since $2n-2>n$ for $n\geq5$, we only need to prove the case $k=5$.
This is done using GAP in Lemma~\ref{lem:gamma5}.
\end{proof}

\begin{remark}
The same proposition, with the same proof, holds for the Grigorchuk group.
The only change is in Lemma~\ref{lem:gamma5}, where one needs to replace the presentation and the endomorphism by those from Equation~\eqref{eq:lpres Grigorchuk}.
\end{remark}

It is convenient to introduce some notation.
For a set $S\subseteq\aut X^*$, we define $\Delta*S=\{(g,g^{-1})\mid g\in S\}$ and $0*S=\{(g,1)\mid g\in S\}$.
We then write $\Delta\Delta*S$ instead of $\Delta*(\Delta*S)$ and so on.

\begin{corollary}\label{cor:lcs}
The group $G$ has finite width. More precisely,
$\gamma_n/\gamma_{n+1}$ is a $(\Z/2\Z)$-vector space, and
\[\operatorname{rank}(\gamma_n/\gamma_{n+1}) = \begin{cases}
  2 & \text{ if }\frac342^i+1\le n\le 2^i\text{ for some }i\in\N,\\
  4 & \text{ otherwise.}\end{cases}\]
\end{corollary}
\begin{proof}
Proposition~\ref{prop:lcs} can be rewritten as
\begin{align*}
\gamma_{2n-1}&=X*\gamma_n,&
\gamma_{2n}&=(\Delta*\gamma_n)(X*\gamma_{n+1}),
\end{align*}
for all $n\geq3$.
From this it is immediate that
\begin{align*}
\gamma_{2n}/\gamma_{2n+1}\simeq\gamma_{2n-1}/\gamma_{2n}\simeq\gamma_n/\gamma_{n+1}
\end{align*}
holds for all $n\geq3$.
Therefore $G$ has finite width.
Even more is true: if $S_n$ is a minimal generating set for $\gamma_n$ modulo $\gamma_{n+1}$, then $0*S_n$ is a minimal generating set for $\gamma_{2n-1}$ modulo $\gamma_{2n}$, and $\Delta*S_n$ is a minimal generating set for $\gamma_{2n}$ modulo $\gamma_{2n+1}$.
Therefore we obtain the following sequence
\begin{align*}
S_1,S_2,&S_3,S_4,\\
&0*S_3,\Delta*S_3,0*S_4,\Delta*S_4,\\
&00*S_3,\Delta 0*S_3,0\Delta*S_3,\Delta\Delta*S_3,00*S_4,\Delta 0*S_4,0\Delta*S_4,\Delta\Delta*S_4,\dots
\end{align*}
of minimal generating sets for $\gamma_n$ (modulo $\gamma_{n+1}$).
Explicitly, we compute the rank and a minimal generating set for $\gamma_n/\gamma_{n+1}$ in Lemma~\ref{lem:explicit lcs}, for $1\le n\le 4$.
It follows that the rank of $\gamma_n/\gamma_{n+1}$ is given by
\begin{align*}
4,4,\;4,2,\;4,4,2,2,\;4,4,4,4,2,2,2,2,\;\dots
\end{align*}
\end{proof}

\section{The Schur Multiplier}\label{sec:Schur}

In this section we compute the second homology group of $G$, also known as the Schur multiplier $M(G)=H_2(G,\Z)$.
Since it has infinite rank, the group $G$ cannot be finitely presented.

Let $\Gamma$ be the group given by the presentation $\langle a,\b,\c,\d\mid a^2,\b^2,\c^2,\d^2\rangle$ and let $\Omega$ be the kernel of the map $\Gamma\to G$ sending generators to generators.
Let $\varphi$ be the endomorphism of $\Gamma$ induced by the one in Theorem~\ref{thm:presentation}.
Write $r_1^{(0)}=[\d^a,\d]$, $r_2^{(0)}=[\d,\c^a\b]$, $r_3^{(0)}=[\d,(\c^a\b)^\c]$ and $r_4^{(0)}=[\c^a\b,\c\b^a]$, and $r_i^{(n)}=\varphi^n(r_i^{(0)})$ for $n\geq0$.
Then the elements $r_i^{(n)}$ form a set of normal generators of $\Omega$ in $\Gamma$.

Using the ``five-term homology sequence"~\cite{robinson:ctg} applied on the extension $1\to\Omega\to\Gamma\to G\to1$, we get the following exact sequence:
\begin{align*}
M(\Gamma)\to M(G)\to\Omega/[\Gamma,\Omega]\to\Gamma/\Gamma'\to G/G'\to1.
\end{align*}
Now $\Gamma/\Gamma'\to G/G'$ is an isomorphism, and the group $\Gamma$ is isomorphic to a free product of four copies of $\Z/2\Z$, whence $M(\Gamma)=0$.
Thus we have an isomorphism $M(G)\simeq\Omega/[\Gamma,\Omega]$.

We shall write $M=M(G)=\Omega/[\Gamma,\Omega]$ and denote by $\overline{r}$ the image in $M$ of an element $r\in\Omega$.

\begin{lemma}\label{lem:schur exponent 2}
The group $M$ is an elementary Abelian $2$-group.
\end{lemma}
\begin{proof}
Since $\Omega\leq\Gamma$, the quotient $\Omega/[\Gamma,\Omega]$ is clearly Abelian.
Moreover, the images of the elements $r_i^{(0)}$ and of their iterates under $\varphi$ generate $M$.
Therefore it is enough to show that $r_i$ is conjugate to its inverse for $i=1,2,3,4$.
A direct computation yields
\begin{align*}
[\d,\d^a]&=[\d^a,\d]^a,&[\c^a\b,\d]&=[\d,\c^a\b]^\d,\\
[(\c^a\b)^\c,\d]&=[\d,(\c^a\b)^\c]^\d,&[\c\b^a,\c^a\b]&=[\c^a\b,\c\b^a]^a,
\end{align*}
thus $M$ is an elementary Abelian $2$-group.
\end{proof}

Let $\psi:\Gamma\to\Gamma\wr C_2$ be the lift of the wreath decomposition of $G$ defined by
\begin{align*}
\psi(a)&=\sigma,&\psi(\b)&=(\c,a),&\psi(\c)&=(a,\d),&\psi(\d)&=(1,\b).
\end{align*}

\begin{lemma}
We have the inclusion $\psi([\Gamma,\Omega])\leq([\Gamma,\Omega]\times[\Gamma,\Omega])\{(r,r^{-1})\mid r\in\Omega\}$.
\end{lemma}
\begin{proof}
We know that the map $\psi$ descends to $G$.
Therefore $\psi(\Omega)\leq\Omega\times\Omega$.
Let $\Sigma\leq\Gamma$ be the subgroup of index $2$ generated by $\{\b,\c,\d,\b^a,\c^a,\d^a\}$.
Then clearly $\psi([\Sigma,\Omega])\leq[\Gamma,\Omega]\times[\Gamma,\Omega]$.
Finally, write $\psi(r)=(r_1,r_2)$ for $r\in\Omega$.
Then $\psi([a,r])=(r_2^{-1}r_1,r_1^{-1}r_2)$, so we proved the claim.
\end{proof}

Let $\Delta$ be the subgroup $\{(\overline{r},\overline{r^{-1}})\mid r\in\Omega\}$ of $M\times M$. This is a subgroup since $M$ is Abelian.
Then the group $(M\times M)/\Delta$ is naturally isomorphic to $M$.
Thus by the above lemma, the map $\psi:\Omega\to\Omega\times\Omega$ induces a map $\Psi:M\to(M\times M)/\Delta\simeq M$.
This endomorphism can be computed as follows.
Let $\overline{r}$ be an element of $M$ and choose a lift $r\in\Omega$.
Then $\Psi(\overline{r})$ is the image of $\psi(r)\in\Omega\times\Omega$ modulo $([\Gamma,\Omega]\times[\Gamma,\Omega])\{(r,r^{-1})\mid r\in\Omega\}$.

\begin{lemma}
We have the isomorphism $M=\ker\Psi\times\Psi(M)$.
\end{lemma}
\begin{proof}
We clearly have the exact sequence $1\to\ker\Psi\to M\to\Psi(M)\to1$.
Since $M$ is elementary Abelian, the extension is a direct sum.
\end{proof}

Let $\overline{R}$ be the subgroup of $M$ generated by $\overline{r_1^{(0)}}$.
The following lemma explains the behaviour of $\Psi$.

\begin{lemma}\label{lem:schur psi shift}
For all $i=1,2,3,4$ and $n>0$, the map $\Psi$ sends $\overline{r_i^{(n)}}$ to $\overline{r_i^{(n-1)}}$, modulo $\overline{R}$.
\end{lemma}
\begin{proof}
Consider $r_i^{(n)}=\varphi(r_i^{(n-1)})$ with $n>0$.
Let $D\leq\Gamma$ be the infinite dihedral group generated by $a$ and $\d$.
Then $\psi(r_i^{(n)})$ equals $(1,r_i^{(n-1)})$ modulo $D\times 1$.
Moreover, since $\psi(r_i^{(n)})$ is in $\Omega\times\Omega$, we must have $\psi(r_i^{(n)})=(1,r_i^{(n-1)})$ modulo $R\times1$ with $R=D\cap\Omega=\langle r_1^{(0)}\rangle$.
This last equality holds because $a$ and $\d$ generate a dihedral group of order $8$ in $G$.
Thus we proved the equality $\Psi(\overline{r_i^{(n)}})=\overline{r_i^{(n-1)}}$ modulo $\overline{R}$.
\end{proof}

\begin{lemma}
The group $N_0\leq M$ generated by $\{\overline{r_1^{(0)}},\overline{r_2^{(0)}},\overline{r_3^{(0)}},\overline{r_4^{(0)}},\overline{r_1^{(1)}}\}$ is in the kernel of $\Psi$.
\end{lemma}
\begin{proof}
For $i=1,2,3,4$, the element $\psi(r_i^{(0)})$ is trivial in $\Gamma\times\Gamma$ as a direct computation shows.
Another straightforward computation yields $\psi(r_1^{(1)})\equiv(r_1^{(0)},r_1^{(0)})$ modulo $[\Gamma,\Omega]\times[\Gamma,\Omega]$, and thus $r_1^{(1)}$ is also in the kernel of $\Psi$.
\end{proof}

\begin{proposition}
The elements $\overline{r_i^{(n)}}$ are all independent in $M$, and $\Psi(r_i^{(n)})$ is non-trivial unless $r_i^{(n)}$ is in $N_0$.
\end{proposition}
\begin{proof}
For $n>0$, let $N_n\leq M$ be the subgroup generated by $\{\overline{r_2^{(n)}},\overline{r_3^{(n)}},\overline{r_4^{(n)}},\overline{r_1^{(n+1)}}\}$.
Note that by Lemma~\ref{lem:gap schur independence}, the group $N_0$ is isomorphic to $(\Z/2\Z)^5$.

Lemma~\ref{lem:schur psi shift} shows that for all $n>0$, we have $\Psi(\overline{r_i^{(n)}})=\overline{r_i^{(n-1)}}$ modulo $\overline{R}$ for $i\in\{2,3,4\}$ and similarly $\Psi(\overline{r_1^{(n+1)}})=\overline{r_1^{(n)}}$ modulo $\overline{R}$.
Thus $\Psi(N_n)=N_{n-1}$, modulo $\overline{R}$.

We now prove by induction that for all $n>0$, the group $N_n$ has rank $4$ and trivial intersection with $N_0\cdots N_{n-1}$.
For the case $n=1$, we note the isomorphism $N_0/\overline{R}\simeq(\Z/2\Z)^4$.
Since $\Psi(N_1)=N_0$ modulo $\overline{R}$, it follows that $N_1$ has rank $4$.
And $N_1$ has trivial intersection with $N_0$ because $\Psi(N_0)=0$ and $\Psi(N_1)\simeq(\Z/2\Z)^4$.

Now suppose the induction hypothesis holds for some $n>0$.
Then $N_n\overline{R}/\overline{R}$ has rank $4$.
We have $\Psi(N_{n+1})=N_n$ modulo $\overline{R}$, and therefore $N_{n+1}$ has rank $4$.
Next $\Psi^{n+1}(N_0\cdots N_n)=0$, whereas $\Psi^{n+1}(N_{n+1})=N_0$ modulo $\overline{R}$.
This proves that $\Psi^{n+1}(N_{n+1})\simeq(\Z/2\Z)^4$ and so $N_{n+1}$ has trivial intersection with $N_0\cdots N_n$.
Therefore the $\overline{r_i^{(n)}}$ are all independent.
\end{proof}
\begin{corollary}
The kernel of $\Psi$ is $N_0\simeq(\Z/2\Z)^5$.
\end{corollary}

The endomorphism $\varphi$ of $\Gamma$ induces an endomorphism on $G$.
Therefore $\varphi(\Omega)\leq\Omega$ and $\varphi$ induces an endomorphism (which we still write $\varphi$) on $M$ defined by $\overline{r}\mapsto\overline{\varphi(r)}$.
Thus one can consider $M$ as a $(\Z/2\Z)[\varphi]$-module.

\begin{lemma}
The endomorphism $\varphi$ is a right inverse of $\Psi$ modulo $\overline{R}$.
\end{lemma}
\begin{proof}
Clearly $\varphi^m(\overline{r_i^{(n)}})=\overline{r_i^{(m+n)}}$.
Using Lemma~\ref{lem:schur psi shift} we deduce $(\Psi\circ\varphi)(\overline{r_i^{(n)}})=\overline{r_i^{(n)}}$ modulo $\overline{R}$, for all $i=1,2,3,4$ and $n\geq0$.
Since the $\overline{r_i^{(n)}}$ generate $M$, the claim is proved.
\end{proof}

We summarize our results in the following

\begin{theorem}\label{thm:schur}
The Schur multiplier $M(G)$ is isomorphic to a direct sum of four copies of $(\Z/2\Z)[\varphi]$.
More precisely, a set of $(\Z/2\Z)[\varphi]$-independent generators of $M(G)$ is $\{\overline{r_1^{(0)}},\overline{r_2^{(0)}},\overline{r_3^{(0)}},\overline{r_4^{(0)}}\}$.
\end{theorem}

\begin{remark}
For the Grigorchuk group, the situation is slightly more subtle.
First, one proves that the Schur multiplier is elementary Abelian ``by hand'', as in Lemma~\ref{lem:schur exponent 2}.

Then, as auxiliary group, one uses $\Gamma=\langle a,b,c,d\mid a^2,b^2,c^2,d^2,bcd\rangle\simeq(\Z/2\Z)*(\Z/2\Z)^2$.
The reason for this is that otherwise $bcd$ would be fixed by $\Psi$.
The Schur multiplier of $\Gamma$ is $\Z/2\Z$, and using the five-term homology sequence we get the extension $1\to\Z/2\Z\to M(G)\to\Omega/[\Gamma,\Omega]\to1$.
The rest of the proof is absolutely similar to the above, and one finally gets $M(G)\simeq\Z/2\Z\times((\Z/2\Z)[\varphi])^2$, which one could pedantically write as $M(G)\simeq((\Z/2\Z)[\varphi]/(\varphi-1))\times((\Z/2\Z)[\varphi])^2$, see~\cite{grigorchuk:bath}.
\end{remark}

\section{GAP Computations}
\begin{lemma}\label{lem:KB}
The L-presentations
\begin{enumerate}
\item\label{item:i}
$\langle a,\b,\c,\d\mid\tilde\varphi\mid a^2,\b^2,\c^2,\d^2,[\d^a,\d],[\d,\c^a\b],[\d,(\c^a\b)^\c],[\c^a\b,\c\b^a]
\rangle$ and
\item\label{item:ii}
$\langle a,\b,\c,\d\mid\tilde\varphi\mid a^2,\b^2,\c^2,\d^2,[\d^a,\d^w],[\d^a,(\c\b^a)^w],[\c\b^a,(\c\b^a)^w]\rangle$,\\
with
$w\in\{1,\c,\c^a,\c\c^a,\c^a\c,\c\c^a\c,\c^a\c\c^a,\c\c^a\c\c^a\}$
\end{enumerate}
define the same group.
\end{lemma}
\begin{proof}
Clearly, all the relations in~\eqref{item:i} are relations in~\eqref{item:ii}.
Therefore it is enough to check that the normal closure of the relations in~\eqref{item:i} together with some of their iterates under $\tilde\varphi$ contains all the relations of~\eqref{item:ii}.
This is done via the following commands using GAP, together with the package KBMAG~\cite{holt:kbmag}:
\begin{verbatim}
gap> f := FreeGroup("a","b","c","d");;
gap> AssignGeneratorVariables(f);
#I  Assigned the global variables [ a, b, c, d ]
gap> frels := [a^2,b^2,c^2,d^2];;
gap> conjs := [One(f),c,c*c^a,c*c^a*c,c*c^a*c*c^a,c^a*c*c^a,c^a*c,c^a];;
gap> irels := Flat(List(conjs, x -> [Comm(d^a,d^x), Comm(d^a,(c*b^a)^x),
> Comm(c^a*b,(c*b^a)^x)]));;
gap> endos:=[GroupHomomorphismByImages(f,f,[a,b,c,d],[c^a,d,b^a,c])];;
gap> k := 2;;
gap> selectedrels := [Comm(d^a,d), Comm(d^a,c*b^a),
> Comm(c^a*b,c*b^a), Comm(d,(c^a*b)^c)];;
gap> g := f / Concatenation(frels, Flat(List([0..k],
> i -> List(selectedrels, r -> r^(endos[1]^i)))));;
gap> rws := KBMAGRewritingSystem(g);;
gap> KnuthBendix(rws);
false
gap> Set(List(irels, x -> ReducedWord(rws, x))) = [One(f)];
#WARNING: system is not confluent, so reductions may not be to normal form.
true
\end{verbatim}
We do not require the Knuth-Bendix procedure to terminate.
Indeed, the command \texttt{KnuthBendix} stops automatically after a few seconds, once it has computed a few thousand new rules.
At this point, the rewriting system already contains enough rules to prove that all the words of \texttt{irels} are congruent to the identity.
\end{proof}

Let $\gamma_n$ be the $n$-th term of the lower central series of $G$.
\begin{lemma}\label{lem:lcs GAP}
The inclusion $[\gamma_3,\{\b,\c,\d\}]\subseteq\gamma_5$ and $[\gamma_4,a]\subseteq\gamma_6$ hold.
\end{lemma}
\begin{proof}
It is enough to work in $G/\gamma_6$, and then to check:
\begin{verbatim}
gap> G := AsLpGroup(GrigorchukEvilTwin);;
gap> pi := NqEpimorphismNilpotentQuotient(G, 5);;
gap> g := Image(pi);;
gap> a := G.1^pi;; b := G.2^pi;; c := G.3^pi;; d := G.4^pi;;
gap> lcs := LowerCentralSeries(g);;
gap> ForAll(GeneratorsOfGroup(lcs[3]), x ->
> ForAll([b,c,d], y -> Comm(x,y) in lcs[5]));
true
gap> ForAll(GeneratorsOfGroup(lcs[4]), x -> Comm(x,a) in lcs[6]);
true
\end{verbatim}
\end{proof}

\begin{lemma}\label{lem:gamma5}
We have the equality $\gamma_5=\gamma_3\times\gamma_3$.
\end{lemma}
\begin{proof}
Let $G=F/\langle R_\infty\rangle^F$ be a presentation of $G$, where $R$ is a finite set, and $R_\infty=\bigcup_{n\geq0}\tilde\varphi^n(R)$ (see Theorem~\ref{thm:presentation}).
We compute a lift $\tilde S$ in $F$ of a set $S$ of normal generators of $\gamma_3$.
By Proposition~\ref{prop:K contains gamma_3}, $\gamma_3$ is a subgroup of $K$, and therefore $\varphi(\gamma_3)=1\times\gamma_3$.
Thus $\gamma_3\times\gamma_3$ is the normal closure of $\varphi(\gamma_3)$, and $\varphi(S)$ is a set of normal generators of $\gamma_3\times\gamma_3$.
Therefore $F/\langle R_\infty\cup\tilde\varphi(\tilde S)\rangle^F$ is a presentation of $G/(\gamma_3\times\gamma_3)$.
This observation allows us to verify that the group $G/(\gamma_3\times\gamma_3)$ has nilpotency class (at most) $4$:
\begin{verbatim}
gap> LoadPackage("fr");;
gap> G := AsLpGroup(GrigorchukEvilTwin);;
gap> pi3 := NqEpimorphismNilpotentQuotient(G, 2);;
gap> iso := IsomorphismFpGroup(Image(pi3));;
gap> iso2 := IsomorphismSimplifiedFpGroup(Image(iso));;
gap> N3 := Image(iso2);;
gap> GeneratorsOfGroup(N3) = List(GeneratorsOfGroup(G), x -> x^(pi3*iso*iso2));
true
gap> f := FreeGroupOfFpGroup(N3);;
gap> Stilde := RelatorsOfFpGroup(N3);;
gap> phi := GroupHomomorphismByImages(f, f,
> [f.1,f.2,f.3,f.4], [f.3^f.1,f.4,f.2^f.1,f.3]);;
gap> R := [f.1^2,f.2^2,f.3^2,f.4^2,Comm(f.4^f.1,f.4),Comm(f.4,f.3^f.1*f.2),
> Comm(f.4,(f.3^f.1*f.2)^f.3),Comm(f.3^f.1*f.2,f.3*f.2^f.1)];;
gap> g := LPresentedGroup(f, [], [phi],
> Concatenation(R, List(Stilde, s -> s^phi)));;
gap> NilpotencyClassOfGroup(NilpotentQuotient(g));
4
\end{verbatim}
The attentive reader will notice that we compute $G/(\gamma_3\times\gamma_3)\simeq F/\langle (R\cup\tilde\varphi(\tilde S))_\infty\rangle^F$, but the equality $\langle(R\cup\tilde\varphi(\tilde S))_\infty\rangle^F=\langle R_\infty\cup\tilde\varphi(\tilde S)\rangle^F$ holds because $\varphi(\gamma_3)$ is a subgroup of $\gamma_3$, and thus $\tilde\varphi(\tilde\varphi(\tilde S))$ is contained in $\langle R_\infty\cup\tilde\varphi(\tilde S)\rangle^F$.

Further, we can now easily prove the inclusion $\gamma_5\geq\gamma_3\times\gamma_3$.
We simply need to check
\begin{verbatim}
gap> pi5 := NqEpimorphismNilpotentQuotient(G, 4);;
gap> pi := GroupHomomorphismByImages(f, Image(pi5),
> [f.1,f.2,f.3,f.4], [G.1^pi5,G.2^pi5,G.3^pi5,G.4^pi5]);;
gap> Set(List(R, r -> r^(phi*pi))) = [One(Image(pi5))];
true
\end{verbatim}
so the image of $\tilde\varphi(\tilde S)$ is trivial in $G/\gamma_5$.
\end{proof}

\begin{lemma}\label{lem:explicit lcs}
The rank of $\gamma_n/\gamma_{n+1}$ is $4,4,4,2$ for $n=1,2,3,4$ respectively.
Moreover, minimal generating sets are given in the following table:
\begin{center}
\begin{tabular}{|c|l|}
\hline
$n$ & generating set for $\gamma_n/\gamma_{n+1}$\\
\hline
1 & $a,\b,\c,\d$\\
2 & $[a,\b],[a,\c],[a,\d],[\b,\c]$\\
3 & $[[a,\b],a],[[a,\b],\c],[[a,\b],\d],[[a,\c],a]$\\
4 & $[[[a,\b],\c],a],[[[a,\b],\d],a]$\\
\hline
\end{tabular}
\end{center}
\end{lemma}
\begin{proof}
We compute in $G/\gamma_5$:
\begin{verbatim}
gap> G := AsLpGroup(GrigorchukEvilTwin);;
gap> pi := NqEpimorphismNilpotentQuotient(G, 4);;
gap> a := G.1^pi;; b := G.2^pi;; c := G.3^pi;; d := G.4^pi;;
gap> g := Image(pi);;
gap> lcs := LowerCentralSeries(g);;
gap> List([2..Length(lcs)], i -> AbelianInvariants(lcs[i-1]/lcs[i]));
[ [ 2, 2, 2, 2 ], [ 2, 2, 2, 2 ], [ 2, 2, 2, 2 ], [ 2, 2 ] ]
gap> s := [[a,b,c,d],
> [Comm(a,b),Comm(a,c),Comm(a,d),Comm(b,c)],
> [LeftNormedComm([a,b,a]),LeftNormedComm([a,b,c]),
>  LeftNormedComm([a,b,d]),LeftNormedComm([a,c,a])],
> [LeftNormedComm([a,b,c,a]),LeftNormedComm([a,b,d,a])]];;
gap> ForAll([1..4], i -> lcs[i] = ClosureGroup(lcs[i+1],s[i]));
true
\end{verbatim}
\end{proof}

For the following lemma the notations are as in Section~\ref{sec:Schur}.
\begin{lemma}\label{lem:gap schur independence}
The elements $\{\overline{r_1^{(0)}},\overline{r_2^{(0)}},\overline{r_3^{(0)}},\overline{r_4^{(0)}},\overline{r_1^{(1)}}\}$ generate a group isomorphic to $(\Z/2\Z)^5$ in the group $\Omega/([\Gamma,\Omega]\varphi^2(\Omega)(\gamma_8(\Gamma)\cap\Omega))$.
\end{lemma}
\begin{proof}
We first define the group \verb+gam+ $=\Gamma/([\Gamma,\Omega]\varphi^2(\Omega))$ as an $L$-presented group using the package NQL~\cite{hartung:nql}.
\begin{verbatim}
gap> f := FreeGroup("a","b","c","d");;
gap> AssignGeneratorVariables(f);;
#I  Assigned the global variables [ a, b, c, d ]
gap> phi := GroupHomomorphismByImages(f, f, [a,b,c,d], [c^a,d,b^a,c]);;
gap> r := [a^2, b^2, c^2, d^2];;
gap> rels := [Comm(d,d^a),Comm(d,c^a*b),Comm(d,(c^a*b)^c),Comm(c^a*b,c*b^a)];;
gap> rels2 := Union(rels, Image(phi, rels));;
gap> r := Union(r, Image(phi^2, rels), ListX([a,b,c,d], rels2, Comm));;
gap> gam := LPresentedGroup(f, [], [phi], r);
<L-presented group on the generators [ a, b, c, d ]>
\end{verbatim}
Then we kill the eighth term of the lower central series of \verb+gam+, and look at the group generated by the images of $\{\overline{r_1^{(0)}},\overline{r_2^{(0)}},\overline{r_3^{(0)}},\overline{r_4^{(0)}},\overline{r_1^{(1)}}\}$.
\begin{verbatim}
gap> AssignGeneratorVariables(gam);;
#I  Assigned the global variables [ a, b, c, d ]
gap> rels := [Comm(d,d^a),Comm(d,c^a*b),Comm(d,(c^a*b)^c),Comm(c^a*b,c*b^a),
> Comm(c,c^(c^a))];;
gap> pi := NqEpimorphismNilpotentQuotient(gam, 7);;
gap> relspi := List(rels, x -> x^pi);;
gap> StructureDescription(Group(relspi));
"C2 x C2 x C2 x C2 x C2"
\end{verbatim}
This is $(\Z/2\Z)^5$, as we wanted to prove.
\end{proof}

\end{document}